\documentclass{amsart}
\usepackage{amssymb,amsthm,amsmath}

\newcommand{\pair}[1]{\langle #1 \rangle}
\newcommand{\wijs}[1]{\tau_{W(#1)}}
\newcommand{\U}{\mathcal{U}}
\newcommand{\cont}{\mathfrak{c}}
\newcommand{\B}{\mathcal{B}}
\newcommand{\A}{\mathcal{A}}

\newcommand{\Q}{\mathbb{Q}}
\newcommand{\I}{\mathcal{I}}
\newcommand{\cl}[2][X]{\mathrm{cl}_{#1}\!\left(#2\right)}

\newtheoremstyle{theorem}
     {11pt}
     {11pt}
     {}
     {}
     {\bfseries}
     {}
     {.5em}
     {\noindent\thmnumber{#2}. \thmname{#1}\thmnote{#3}}

\theoremstyle{theorem}

\newtheorem{lemma}{Lemma}[section]
\newtheorem{propo}[lemma]{Proposition}
\newtheorem{coro}[lemma]{Corollary}
\newtheorem{ex}[lemma]{Example}
\newtheorem{thm}[lemma]{Theorem}
\newtheorem{ques}[lemma]{Question}

\newenvironment{ex2}{\noindent\!\!}{\hfill\qed\vspace{1ex}}

\title{Wijsman hyperspaces of non-separable metric spaces}
\author[R. Hern\'andez-Guti\'errez]{Rodrigo Hern\'andez-Guti\'errez}
\email[R. Hern\'andez-Guti\'errez]{rodhdz@yorku.ca}
\author[P. Szeptycki]{Paul Szeptycki}
\email[P. Szeptycki]{szeptyck@yorku.ca}

\address{Department of Mathematics and Statistics, York University, Toronto, ON M3J 1P3, Canada}

\date{\today}

\begin{document}
 
\begin{abstract}
Given a metric space $\pair{X,\rho}$, consider its hyperspace of closed sets $CL(X)$ with the Wijsman topology $\wijs{\rho}$. It is known that $\pair{CL(X),\wijs{\rho}}$ is metrizable if and only if $X$ is separable and it is an open question by Di Maio and Meccariello whether this is equivalent to $\pair{CL(X),\wijs{\rho}}$ being normal. In this paper we prove that if the weight of $X$ is a regular uncountable cardinal and $X$ is locally separable, then $\pair{CL(X),\wijs{\rho}}$ is not normal. We also solve some questions by Cao, Junnilla and Moors regarding isolated points in Wijsman hyperspaces.
\end{abstract}

\maketitle

\section{Introduction}

Given a metric space $\pair{X,\rho}$, consider the \emph{hyperspace} $CL(X)$ of all closed non-empty subsets of $X$ with the Wijsman topology $\wijs{\rho}$ (defined in the next section). Perhaps the most surprising fact about the Wijsman topology is that it depends not only on the topology of the base space $X$ but also on the specific metric $\rho$ used to generate it (see, for example, Corollary \ref{changes} below). This may be a reason why the structure of the Wijsman topology is much more intricate than the more known and widely studied Vietoris topology. 

A specific topological property we will consider is normality. The classification of normal Vietoris hyperspaces is now a classic result of Veli\v cko (\cite{velicko}) from 1975. However, the corresponding characterization for Wijsman hyperspaces is still an open question. It is known that the Wijsman hyperspace $\pair{CL(X),\wijs{\rho}}$ is metrizable if and only if $X$ is separable (see \cite[Theorem 2.1.5]{beer}). In 1998, Di Maio and Meccariello asked the following.

\begin{ques}\cite{dimaio-meccariello}\label{normality}
Let $\pair{X,\rho}$ be a metric space. Is it true that if $\pair{CL(X),\wijs{\rho}}$ is normal then $X$ is separable?
\end{ques}

Most of the work done so far points to the answer to this question being in the affirmative. Two of the most relevant results are the following.

\begin{thm}\cite[Theorem 1.2]{chaber-pol}\label{nlc}
Let $X$ be a metrizable space such that the set of points in $X$ with no compact neighborhood has weight $\kappa$. Then for any metric $\rho$ compatible with the topology of $X$, the space $\omega\sp\kappa$ embeds as a closed subspace of $\pair{CL(X),\wijs{\rho}}$. Thus, if $\kappa>\omega$, $\pair{CL(X),\wijs{\rho}}$ is not normal.
\end{thm}

\begin{thm}\cite[Theorem 3.1]{cao-jun-hernorm}\label{herednorm}
Let $\pair{X,\rho}$ be a non-separable metric space. Then $CL(X)\setminus\{X\}$, given the subspace topology of $\wijs{\rho}$, has a closed copy of the Dieudonn\'e plank $\omega_1\times(\omega_1+1)$. Thus, $\pair{CL(X),\wijs{\rho}}$ is not hereditarily normal.
\end{thm}

The major contribution of this paper is to approach the solution of Question \ref{normality} from a different angle. Our main result is the following.

\begin{thm}\label{main}
Let $\kappa$ be a regular uncountable cardinal. If $\pair{X,\rho}$ is a locally separable metric space of weight $\kappa$, then $\pair{CL(X),\wijs{\rho}}$ is not normal.
\end{thm}

Notice that locally compact metrizable spaces are locally separable. Thus, in some way, we are solving Question \ref{normality} in a case that is exactly the opposite of the case considered in Theorem \ref{nlc}. Unfortunately, the proof of Theorem \ref{main} is not as direct as embedding a well-known non-normal space as a closed set. We were able to obtain an embedding theorem that, nevertheless, is not strong enough to assure non-normality; see Theorem \ref{embedding} below.

In \cite{cao-jun-moors}, Cao, Junnila and Moors ask some questions regarding isolated points in Wijsman hyperspaces. In the last section of this paper, we answer them.

\section{Preliminaries}

For the background on general topology see \cite{eng}; for the set set-theoretic background see \cite{kunen-set-theory-2011}. Let $\pair{X,\rho}$ be a metric space. For $x\in X$ and $\epsilon\in(0,\infty)$, we define the open and closed balls
\begin{eqnarray}
B_\rho(x,\epsilon) & = & \{y\in X:\rho(x,y)<\epsilon\}, \nonumber\\
D_\rho(x,\epsilon) & = & \{y\in X:\rho(x,y)\leq\epsilon\}. \nonumber
\end{eqnarray} 
The metric $\rho$ is an ultrametric if $\rho(x,y)\leq\max{\{\rho(x,z),\rho(z,y)\}}$ for all $x,y,z\in X$. In particular, this means that every triangle in $X$ is isosceles with each of its two equal sides greater than the remaining one. The following is well-known and easy to prove.

\begin{lemma}\label{ultralemma}
Let $\pair{X,\rho}$ be an ultrametric space and let $\B=\{B_\rho(x,\epsilon):x\in X,\epsilon\in(0,\infty)\}$. Then
\begin{itemize}
\item[(a)] $\B$ is a base of clopen subsets of $X$, and
\item[(b)] if $B_0,B_1\in\B$ and $B_0\cap B_1\neq\emptyset$, then $B_i\subset B_{1-i}$ for some $i\in\{0,1\}$.
\end{itemize}
\end{lemma}

The Wijsman topology $\wijs{\rho}$ on $CL(X)=\{A\subset X :A\textrm{ is closed and non-empty}\}$ defined by the metric $\rho$ is the smallest topology such that the family of functionals $\{\rho(\_\ ,p):p\in X\}$ is continuous, where $\rho(A,p)=\inf\{\rho(x,p):x\in A\}$ for each $p\in X$ and $A\in CL(X)$. Given $x\in X$ and $\epsilon>0$, let
\begin{eqnarray}
\U\sp+(x,\epsilon)& = &\{A\in CL(X):\rho(A,x)>\epsilon\},\textrm{ and}\nonumber\\
\U\sp-(x,\epsilon)& = &\{A\in CL(X):\rho(A,x)<\epsilon\}.\nonumber
\end{eqnarray}
The collection $\{\U\sp+(x,\epsilon):x\in X,\epsilon>0\}\cup\{\U\sp-(x,\epsilon):x\in X,\epsilon>0\}$ is taken as the canonical subbase for $\wijs{\rho}$. Another subbase consists of all sets of the form
$$
\U(x,\delta,\epsilon)=\{A\in CL(X):\delta<\rho(A,x)<\epsilon\}
$$
such that $x\in X$ and $\delta<\epsilon$. We will need the following observation.

\begin{lemma}\label{basedense}
Let $\pair{X,\rho}$ be a metric space. If $D\subset X$ is a dense set and $\Q$ is the set of rational numbers, then $\{\U(x,q,r):x\in D,q,r\in\Q,q<r\}$ is a base for $\pair{CL(X),\wijs{\rho}}$.
\end{lemma}

Notice that from Lemma \ref{basedense} we can infer that the weight of $\pair{CL(X),\wijs{\rho}}$ is less or equal to the density (equivalently, weight) of $X$.

\section{Normality of hyperspaces}

In this section we will give the proof of Theorem \ref{main}. First we need a result that allows us to partition locally separable metrizable spaces into clopen separable pieces. It easily follows from the proof of \cite[5.1.27]{eng} but we give the proof for the sake of completeness.

\begin{lemma}\label{decomposition}
If $X$ is a locally separable metrizable space of weight $\kappa>\omega$, then there is a clopen partition $X=\bigcup\{K_\alpha:\alpha<\kappa\}$ where $K_\alpha$ is a separable space for all $\alpha<\kappa$.
\end{lemma}
\begin{proof}
By paracompactness of $X$, there is a locally finite open cover $\U$ of $X$ consisting of separable open subsets. For each $x\in X$, define $U(x)$ as the set of all points $y\in X$ such that there exist $U_0,\ldots,U_m\subset\U$ such that $x\in U_0$, $y\in U_m$ and $U_i\cap U_{i+1}\neq\emptyset$ every time $i<m$. Clearly, $\{U(x):x\in X\}$ is a partition of $X$ into clopen pieces.

Since every separable metric space is Lindel\"of, by the fact that $\U$ is locally finite it easily follows that for all $U\in\U$ the set $\{V\in\U:U\cap V\neq\emptyset\}$ is countable. Given $x\in X$, if we recursively define $\U(x,0)=\{U\in\U:x\in U\}$ and $\U(x,i+1)=\{U\in\U:\exists V\in\U(x,i)\ (U\cap V\neq\emptyset)\}$ for $i<\omega$, then it follows that $|\U(x,i)|\leq\omega$ for all $i<\omega$. Also, $U(x)=\bigcup\{\bigcup\U(x,i):i<\omega\}$ so this set is a countable union of separable open subspaces, thus it is separable.

Notice that the fact that the space $X$ has weight $\kappa$ implies that $|\{U(x):x\in X\}|=\kappa$. Then let $\{K_\alpha:\alpha<\kappa\}$ be an enumeration of $\{U(x):x\in X\}$.
\end{proof}

The closed sets we will use to prove non-normality of the Wijsman hyperspace will be chains of decreasing clopen subsets of the base space. Let us first give some properties of this type of sets.

\begin{lemma}\label{decreasing}
Let $\pair{X,\rho}$ be a metric space. Assume that there exists a limit ordinal $\kappa$ and a set $\A=\{A_\alpha:\alpha<\kappa\}\subset CL(X)$ such that
\begin{itemize}
\item[(i)] $A_\beta\subsetneq A_\alpha$ whenever $\alpha<\beta<\kappa$, 
\item[(ii)] if $\gamma<\kappa$ is a limit ordinal, then $A_\gamma=\bigcap\{A_\alpha:\alpha<\gamma\}$, and
\item[(iii)] $\bigcap\A=\emptyset$.
\end{itemize}
Then $\A$ is a closed subset of $\pair{CL(X),\wijs{\rho}}$ and its subspace topology is finer than the order topology induced by the enumeration.
\end{lemma}
\begin{proof}
In order to simplify the proof, let us assume that $A_0=X$, clearly we will not lose generality from this assumption.

First to show that $\A$ is closed, let $B\in CL(X)\setminus\A$. Notice that this implies that there is $\gamma<\kappa$ such that $\gamma+1=\min\{\alpha<\kappa:B\not\subset A_\alpha\}$. Let $x\in B\setminus A_{\gamma+1}$, then there exists $\epsilon_0>0$ such that $B_\rho(x,\epsilon_0)\cap A_{\gamma+1}=\emptyset$. Since $B\neq A_\gamma$, there is a point $y\in A_\gamma\setminus B$ so let $\epsilon_1>0$ be such that $D_\rho(y,\epsilon_1)\cap B=\emptyset$. Then $B\in\U\sp{-}(x,\epsilon_0)\cap \U\sp{+}(y,\epsilon_1)$ and $\U\sp{-}(x,\epsilon_0)\cap\U\sp{+}(y,\epsilon_1)\cap\A=\emptyset$. 

Let $\beta<\kappa$ and $x\in A_\beta\setminus A_{\beta+1}$. Then there exists $\epsilon>0$ such that $D_\rho(x,\epsilon)\cap A_{\beta+1}=\emptyset$. Thus, $\U\sp{+}(x,\epsilon_x)\cap\A=\{A_\alpha:\beta<\alpha<\kappa\}$. This implies that all final segments are open. Moreover, in this same situation, $\U\sp{-}(x,\epsilon)\cap\A=\{A_\alpha:\alpha\leq\beta\}$. When $\gamma=\beta+1$, this means that the initial segment $\{A_\alpha:\alpha<\gamma\}$ is open. And when $\gamma<\kappa$ is a limit cardinal, then
$$
\{A_\alpha:\alpha<\gamma\}=\bigcup_{\beta<\gamma}{\{A_\alpha:\alpha\leq\beta\}}
$$
is also an open initial segment. This shows that the subspace topology of $\A$ contains all order-open sets.
\end{proof}

The following technical result will be used in our proof of Theorem \ref{main} but we will also be able to use it in Theorem \ref{embedding} below so we keep it as a separate lemma.

\begin{lemma}\label{sorgenfrey}
Let $\kappa$ be an infinite cardinal of uncountable cofinality and let $f:\kappa\to[0,\infty)$ be a function. Assume further that $f(0)=0$ and there exists a real number $\delta>0$ such that $f(\alpha)>\delta$ for each $0<\alpha<\kappa$. Then there exists $r\in(0,\infty)$ such that for every $\epsilon>0$ the set $\{\alpha<\kappa:f(\alpha)\in[r,\epsilon)\}$ is of cardinality $\kappa$ and $\{\alpha<\kappa:f(\alpha)<r\}$ is of cardinality strictly less than $\kappa$.
\end{lemma}
\begin{proof}
Let $U$ be the set of those points $t\in(0,\infty)$ such that there is $\epsilon>0$ such that $\{\alpha<\kappa:f(\alpha)\in[t,t+\epsilon)\}$ is of cardinality less than $\kappa$. Notice that $U$ is an open set of $(0,\infty)$ with the Sorgenfrey line topology.

Assume that $U=(0,\infty)$, we will reach a contradiction. Since the Sorgenfrey line is Lindel\"of, by the definition of $U$ it is possible to find countable sets $\{t_n:n<\omega\}\cup\{\delta_n:n<\omega\}\subset(0,\infty)$ such that $|\{\alpha<\kappa:f(\alpha)\in[t_n,t_n+\delta_n]\}|<\kappa$ for each $n<\omega$ and $(0,\infty)=\bigcup\{[t_n,t_n+\delta_n):n<\omega\}$. But from the fact that $\kappa$ is of uncountable cofinality and $\kappa\setminus\{0\}=\bigcup_{n<\omega}{\{\alpha<\kappa:f(\alpha)\in[t_n,t_n+\delta_n]\}}$ we obtain that $|\kappa\setminus\{0\}|<\kappa$, a contradiction.

Thus $U$ is a proper subset of $(0,\infty)$. Then define $r=\inf{([0,\infty)\setminus U)}$, notice that $r\geq\delta>0$ so $r\in (0,\infty)\setminus U$. From this the lemma follows easily.
\end{proof}

We finally have everything we need to give the proof of our main result.

\begin{proof}[{\bf Proof of Theorem \ref{main}}]
First, let $X=\bigcup\{K_\alpha:\alpha<\kappa\}$ be a decomposition as in Lemma \ref{decomposition}. Let $D$ be a dense set of $X$ such that $D\cap K_\alpha$ is countable for all $\alpha<\kappa$. We will apply Lemma \ref{basedense}, and consider the base $\{\U(x,q,r):x\in D,q,r\in\Q,q<r\}$.

For each $x\in D$, let $f_x:\kappa\to[0,\infty)$ be defined as $f_x(\alpha)=\rho(x,K_\alpha)$ for all $\alpha<\kappa$. Notice that if $x\in K_\beta$ for some $\beta<\kappa$, then there is $\delta_x>0$ such that $D_\rho(x,\delta_x)\subset K_\beta$. In this case, $f_x(\alpha)>\delta_x$ if $\alpha\neq\beta$. Thus, we may apply Lemma \ref{sorgenfrey} and obtain $r(x)\in(0,\infty)$ such that $|\{\alpha<\kappa:f_x(\alpha)\in[r(x),\epsilon)\}|=\kappa$ for all $\epsilon>0$ and $|\{\alpha<\kappa:f_x(\alpha)<r(x)\}|<\kappa$.

We shall construct a closed set $A$ with the following properties:
\begin{itemize}
\item[(a)] for each $\alpha<\kappa$, either $K_\alpha\subset A$ or $A\cap K_\alpha=\emptyset$,
\item[(b)] both $\{\alpha<\kappa:K_\alpha\subset A\}$ and $\{\alpha<\kappa:A\cap K_\alpha=\emptyset\}$ have cardinality $\kappa$, and
\item[(c)] for every $x\in D$ and $n<\omega$, the set $\{\alpha<\kappa:K_\alpha\subset A, f_x(\alpha)\in[r(x),r(x)+1/(n+1))\}$ is of cardinality $\kappa$.
\end{itemize}

First, enumerate $D=\{d_\alpha:\alpha<\kappa\}$. Then choose different ordinals 
$$
\Gamma=\{t(\alpha,\beta,n,i):\alpha\leq\beta<\kappa,n<\omega,i\in 2\}\subset\kappa
$$
such that $f_{d_\alpha}(t(\alpha,\beta,n,i))\in[r(d_\alpha),r(d_\alpha)+1/(n+1))$ every time $\alpha\leq\beta<\kappa$, $n<\omega$ and $i\in 2$. This is not hard to do by a recursion of length $\kappa$ such that in step $\gamma<\kappa$ we choose $\{t(\alpha,\gamma,n,i):\alpha\leq\gamma,n<\omega,i\in 2\}$ all different from the ordinals chosen in previous steps. Finally, let
$$
A=\bigcup\{K_\alpha:\alpha\in \Gamma\},
$$
which clearly has the properties we wanted.

Now we can define the closed sets that cannot be separated. For all $\beta<\kappa$, let $A_\beta=A\cap \left(\bigcup\{K_\alpha:\beta\leq\alpha\}\right)$ and $B_\beta=\bigcup \{K_\alpha:\beta\leq\alpha\}$. Then $\A=\{A_\alpha:\alpha<\kappa\}$ and $\B=\{B_\alpha:\alpha<\kappa\}$ are disjoint non-empty closed subsets of $\pair{CL(X),\wijs{\rho}}$ by property (b) and Lemma \ref{decreasing}. Let us start by emphasizing the following property.

\vskip6pt
\noindent{$(\ast)$} For each $x\in D$, there exists $\gamma<\kappa$ such that if $\gamma<\alpha<\kappa$, then $\rho(x,A_\alpha)=\rho(x,B_\alpha)=r(x)$.
\vskip6pt

By the definition of $r(x)$, $\{\alpha<\kappa:f_x(\alpha)<r(x)\}\subset\gamma$ for some $\gamma<\kappa$. So if $\gamma<\alpha<\kappa$, then $f_x(\alpha)\in[r(x),\infty)$. This implies that $r(x)\leq\rho(x,B_\alpha)$ every time $\gamma<\alpha<\kappa$. Moreover, by property $(c)$ in the definition of $A$, if $n<\omega$ and $\gamma<\alpha<\kappa$ there is $\beta<\kappa$ with $\alpha<\beta$, $K_\beta\subset A$ and $f_x(\beta)<r(x)+1/(n+1)$; this implies that $\rho(x,A_\alpha)\leq r(x)$. Hence, $r(x)\leq\rho(x,B_\alpha)=\rho(x,A_\alpha)\leq r(x)$ if $\gamma<\beta<\kappa$. Property $(\ast)$ is thus proved.

Now let us assume that $\A$ and $\B$ can be separated, we will then arrive to a contradiction. So there are disjoint open sets $U$ and $V$ such that $\A\subset U$ and $\B\subset V$. Let $\I$ be the set of non-empty open intervals with endpoints in $\Q$; notice that $\I$ is countable. Then for each $\alpha<\kappa$ there are finite subsets $S_\alpha$ and $T_\alpha$ of $D$, functions $\phi_\alpha:S_\alpha\to\I$ and $\psi_\alpha:T_\alpha\to \I$ such that
\begin{eqnarray}
A_\alpha&\in&\bigcap\{\U(x,p,q):x\in S_\alpha,\phi_\alpha(x)=(p,q)\}\subset U,\textrm{ and}\nonumber\\
B_\alpha&\in&\bigcap\{\U(x,p,q):x\in T_\alpha,\psi_\alpha(x)=(p,q)\}\subset V.\nonumber
\end{eqnarray}

By the regularity of $\kappa$, we may apply the Pressing Down Lemma (\cite[Lemma III.6.14]{kunen-set-theory-2011}) so there exists $\lambda<\kappa$ and $\Lambda\in[\kappa\setminus\lambda]\sp{\kappa}$ such that if $\beta\in\Lambda$ then $\{\alpha<\beta:T_\beta\cap K_\alpha\neq\emptyset\}\subset\lambda$. Since $|\bigcup\{D_\alpha:\alpha<\lambda\}|<\kappa$ and we are only dealing with finite sets, we may assume that for each $\beta,\gamma\in\Lambda$ and $\alpha<\lambda$ then $T_\beta\cap K_\alpha=T_\gamma\cap K_\alpha$. Call $T=T_\beta\cap(\bigcup\{D_\alpha:\alpha<\lambda\})$ for any $\beta\in\Lambda$. Then $\{T_\alpha:\alpha\in\Lambda\}$ forms a $\Delta$-system with root $T$ and has the additional property that $T_\beta\setminus T\subset\bigcup\{K_\alpha:\beta\leq\alpha\}$ for each $\beta\in\Lambda$.

Now apply the $\Delta$-system Lemma (\cite[Lemma III.2.6]{kunen-set-theory-2011}) so we may assume that $\{S_\alpha:\alpha\in\Lambda\}$ forms a $\Delta$-system with root $S$. Since $\I$ is countable, we may refine again and assume that there are functions $\phi:S\to\I$ and $\psi:T\to\I$ such that $\phi=\phi_\alpha\!\!\restriction_T$ and $\psi=\psi_\alpha\!\!\restriction_S$ for all $\alpha\in\Lambda$.

By property $(\ast)$, it is possible to find $\mu_0\in\Lambda$ such that for every $x\in S\cup T$ and $\mu_0<\alpha<\kappa$, $\rho(x,A_\alpha)=\rho(x,B_\alpha)=r(x)$. Now, for each $x\in S_{\mu_0}\setminus S$, let $z_x\in A_{\mu_0}$ be such that $\rho(x,z_x)\in\phi(x)$. Again by property $(\ast)$, there is $\mu_1\in\Lambda$ with $\mu_0<\mu_1$ such that for every $x\in S_{\mu_1}\setminus S$ and $\mu_1<\alpha<\kappa$, $\rho(x,B_\alpha)=r(x)$ and $\rho(z_x,B_\alpha)=r(z_x)$. Notice that by our construction $T_{\mu_1}\setminus T\subset B_{\mu_1}$.

Now, let $x\in S\cap T$. By the definition of $\mu_0$, $\phi(x)=\phi_{\mu_1}(x)$ and $\psi(x)=\psi_{\mu_1}(x)$ are intervals that contain the point $r(x)=\rho(x,A_{\mu_1})$. Thus, there exists $y_x\in A_{\mu_1}$ such that $\rho(x,y_x)\in\phi(x)\cap\psi(x)$. Similarly, if $x\in S\setminus T$ or $x\in T\setminus S$, there is $y_x\in A_{\mu_1}$ such that $\rho(x,y_x)\in\phi(x)$ or $y_x\in B_{\mu_1}$ such that $\rho(x,y_x)\in\psi(x)$, respectively.

Let $F=\{y_x:x\in S\cup T\}\cup\{z_x:x\in S_{\mu_0}\setminus S\}\cup (T_{\mu_1}\setminus T)$. Then $F$ is a non-empty finite (thus, closed) subset of $X$. We will now argue that $F\in U\cap V$, which is the contradiction we are looking for.

First we prove that $F\in U$. Start by considering a point $s\in S_{\mu_0}$ and let $\phi(s)=(p,q)$. By property (c) in the construction of $A$, it is easy to see that $\rho(s,A_{\mu_0})\leq r(s)$ so $p<r(s)$.
\vskip6pt
\noindent{\it Case 1. }$s\in S$
\vskip6pt
Since $F\subset B_{\mu_0}$, by the definition of $\mu_0$ for every $k\in F$, $\rho(s,k)\geq r(s)$. Notice further that $y_s\in F$ is such that $\rho(s,y_s)\in (p,q)$. Thus $\rho(s,F)\in(p,q)$.
\vskip6pt
\noindent{\it Case 2. }$s\notin S$
\vskip6pt
By the definition of $\mu_1$, given $k\in\{y_x:x\in S\cup T\}\cup (T_{\mu_1}\setminus T)$ then $\rho(s,k)\geq r(s)$. If $k\in\{z_x:x\in S_{\mu_0}\setminus S\}$, then $k\in A_{\mu_0}$ so $\rho(s,k)\geq\rho(s,A_{\mu_0})\in(p,q)$. Notice further that $z_s\in F$ is such that $\rho(s,z_s)\in (p,q)$. So we obtain that $\rho(s,F)\in(p,q)$.
\vskip6pt
So both in Cases 1 and 2 we obtain that $F\in\U(s,p,q)$. Thus, by considering all possible $s\in S$, we obtain that  $F\in U$.
\vskip6pt
Now let us prove that $F\in V$. Take $t\in T_{\mu_1}$ and let $\phi(t)=(p,q)$. By the definition of $r(t)$, $\rho(t,B_{\mu_1})\leq r(t)$ so $p<r(t)$.
\vskip6pt
\noindent{\it Case 1. }$t\in T$
\vskip6pt
Again, $F\subset B_{\mu_0}$ so for every $k\in F$, $\rho(t,k)\geq r(s)$. Also, $y_t\in F$ is such that $\rho(t,y_t)\in (p,q)$. Thus $\rho(t,F)\in(p,q)$.
\vskip6pt
\noindent{\it Case 2. }$t\notin T$
\vskip6pt
In this case, $t\in B_{\mu_1}$ so $\rho(t,B_{\mu_1})=0$ which means that $0\in(p,q)$. Then $t\in T_{\mu_1}\setminus T$, so $t\in F$ itself witnesses that $\rho(t,F)=0\in(p,q)$.
\vskip6pt
So in Cases 1 and 2 we obtain that $F\in U(t,p,q)$. Thus, by considering all possible $t\in T$, we obtain that  $F\in V$. This implies that $F\in U\cap V$ which is a contradiction. Thus, $\pair{CL(X),\wijs{\rho}}$ is not normal and we have finished the proof.
\end{proof}

As we have mentioned before, all locally compact metrizable spaces are locally separable.

\begin{coro}
Let $\kappa$ be a regular uncountable cardinal. If $\pair{X,\rho}$ is a locally compact metric space of weight $\kappa$, then the space $\pair{CL(\kappa),\wijs{\rho}}$ is not normal.
\end{coro}

In particular, 

\begin{coro}
Let $\kappa$ be a regular uncountable cardinal. If $\pair{X,\rho}$ is a discrete metric space of cardinality $\kappa$, then the space $\pair{CL(\kappa),\wijs{\rho}}$ is not normal.
\end{coro}

Also, the class of locally separable metrizable spaces is strictly larger than the class of locally compact metrizable spaces.

\begin{ex}
If $J(\omega)$ is the hedgehog of $\omega$ spines (see \cite[4.1.15]{eng}) and $\kappa$ is an uncountable cardinal, then $J(\omega)\times\kappa$ is a locally separable metrizable space that is not locally compact and has weight $\kappa$.
\end{ex}

Now we will make some comments about Wijsman hyperspaces of metrizable spaces that are not necessarily locally separable. Let $X$ be a metrizable space with its weight a regular cardinal $\kappa$ and assume that 
$$
R=\{x\in X:x\textrm{ has no separable neighborhood}\}
$$
is non-empty. Clearly, no point of $R$ has a compact neighborhood. If $R$ is non-separable, by Theorem \ref{nlc}, then $\pair{CL(X),\wijs{\rho}}$ is not normal for any compatible metric $\rho$ of $X$. Thus, we are left with the case when $R$ is separable.

Fix some metric $\rho$ on $X$. The subspace $Y=X\setminus R$ is locally separable so by Lemma \ref{decomposition} it is the union of a pairwise disjoint and clopen family of separable subsets $\{K_\alpha:\alpha<\kappa\}$. Now consider the function
$$
e:CL(Y)\to CL(X)
$$
defined by $f(A)=A\cup R$. Clearly, this function is one-to-one. Recall that by Lemma \ref{basedense}, the Wijsman topology of both $CL(X)$ and $CL(Y)$ is determined by a dense set $D$ which we may choose to be disjoint from $R$ (because this set is nowhere dense). This easily implies that the Wijsman topology of the image $e[CL(Y)]$ coincides with the subspace topology as a subset of $CL(X)$. Thus, $e$ is an embedding.

Let $C\in CL(X)$ such that $C\in\cl[CL(X)]{e[CL(Y)]}\setminus e[CL(Y)]$. If there is $x\in R\setminus C$ then $\U\sp+(x,\epsilon)$ is a neighborhood of $C$ that misses $e[CL(Y)]$, where $\epsilon=\frac{1}{2}\rho(x,R)$. So in fact $R\subset C$. Also, if $C\setminus R\neq\emptyset$, then $C=e(C\setminus R)$. Thus, $C\subset R$ so $R=C$. This means that the only possible limit point of $e[CL(Y)]$ is $R$.

Now let $\A$ and $\B$ be the closed sets of $CL(Y)$ constructed in the same way as in the proof of Theorem \ref{main}. Consider the same notation as in that proof. If $R\notin\cl[CL(X)]{e[\B]}$, it is possible to separate $e[\B]$ from $\{R\}$ by open sets. Then it is easy to see that there are finitely many points $x_0,\ldots,x_m\in D$ such that $\rho(x_i,B_\alpha)<\rho(x_i,R)$ for $i\leq m$. Then $\rho(x_i,A_\alpha)<\rho(x_i,R)$ for $i\leq m$, which means that $e[\A]$ can also be separated from $\{R\}$. In other words, both $\A$ and $\B$ map to closed sets under $e$. Since $\A$ and $\B$ cannot be separated by open sets in $Y$, then its images cannot be separated in $X$. This proves that in this specific case, $\pair{CL(X),\wijs{\rho}}$ is not normal. However, we don't know how to do the case when $R\in\cl[CL(X)]{e[\B]}$.

It is worth remarking that, for example, if $J(\kappa)$ is the hedgehog with $\kappa$ spines (see \cite[4.1.15]{eng}) and if $I_\alpha$ denotes the $\alpha$'th spine with $0$ the common point to all the spines, then in the above decomposition we have the $I_\alpha=K_\alpha$ and $R$ consists of the single common point $0$ to all the spines. In the usual metric on the hedgehog, it is not hard to prove that the Wijsman hyperspace is not normal, even though $R$ is a limit of $\cl[CL(J(\kappa))]{e[\B]}$. And on the other hand, there is a topologically equivalent metric on $CL(J(\kappa))$ in which $R$ is not a limit point of $\cl[CL(J(\kappa))]{e[\B]}$. 

Now let us note that in contrast to the proofs of Theorems \ref{nlc} and \ref{herednorm}, our proof  did not give an embedding of a well-known non-normal space as a closed subspace of the Wijsman hyperspace. However, in some instances we can embed certain ordinals as closed subspaces.	 The decreasing sequences from Lemma \ref{decreasing} are almost ordinals and in fact with a bit more work we can construct such sequences so that resulting closed subspace is homeomorphic to the order type of the sequence with respect to the order topology. 

\begin{thm}\label{embedding}
Let $\kappa$ be a regular uncountable cardinal. If $\pair{X,\rho}$ is a locally separable metric space of weight $\kappa$, $\kappa$ with the order topology embeds as a closed subset of the space $\pair{CL(X),\wijs{\rho}}$. 
\end{thm}
\begin{proof}
Let $X=\bigcup\{K_\alpha:\alpha<\kappa\}$ be a partition as in Lemma \ref{decomposition}. Let $D$ be a dense set of $X$ such that $D_\alpha= D\cap K_\alpha$ is countable for all $\alpha<\kappa$. Notice that by Lemma \ref{basedense}, for all our arguments it is enough to consider the open sets defined by elements of $D$.

The proof will be carried out by recursively constructing a partition $\{X_\alpha:\alpha<\kappa\}$ of $X$, where $X_\beta$ is a union of elements of $\{K_\alpha:\alpha<\kappa\}$ for each $\beta<\kappa$. Once we have this partition, we define $A_\beta=\bigcup\{X_\alpha:\beta\leq\alpha<\kappa\}$ for all $\beta<\kappa$. By Lemma \ref{decreasing}, the set $\A=\{A_\alpha:\alpha<\kappa\}$ will be closed and its topology will be finer than the order topology on $\kappa$. Our objective is to construct $\{X_\alpha:\alpha<\kappa\}$ in such a way that the subspace topology coincides with the order topology on $\kappa$. 

Let us comment informally on how to achieve this. By Lemma \ref{decreasing}, we only have to prove that for each limit ordinal $\lambda$, $\lim_{\beta\rightarrow\lambda}A_\beta=A_\lambda$ with the Wijsman topology. This amounts to proving that for each $x\in D$, the function $\alpha\mapsto\rho(x,A_\alpha)$ is continuous at limit ordinals. If we make sure that changes in the values of this function occur at successor stages, we will have no continuity problems at limits. In order to do this, we will first prove that for each $x\in D$ there is a maximal possible value of the function $\alpha\mapsto\rho(x,A_\alpha)$ (see property $(\ast)$ below).

For each $x\in D$, let $f_x:\kappa\to[0,\infty)$ be defined as $f_x(\alpha)=\rho(x,K_\alpha)$ for all $\alpha<\kappa$. Notice that if $x\in D_\beta$ for some $\beta<\kappa$, then there is $\delta_x>0$ such that $D_\rho(x,\delta_x)\subset K_\beta$. In this case, $f_x(\alpha)>\delta_x$ if $\alpha\neq\beta$. Thus, we may apply Lemma \ref{sorgenfrey} and obtain $r(x)\in(0,\infty)$ such that $|\{\alpha<\kappa:f_x(\alpha)\in[r(x),\epsilon)\}|=\kappa$ for all $\epsilon>0$ and $|\{\alpha<\kappa:f_x(\alpha)<r(x)\}|<\kappa$. So define $T(x)=\{\alpha<\kappa:f_x(\alpha)<r(x)\}$. The following property of $r(x)$ will be important for our construction.

\vskip6pt
\noindent{$(\ast)$}  If $x\in D$ and $A\in CL(X)$ is such that $|\{\alpha<\kappa:K_\alpha\not\subset A\}|<\kappa$, then $\rho(x,A)\leq r(x)$.
\vskip6pt

Now we can finally construct the partition $\{X_\alpha:\alpha<\kappa\}$. Instead of constructing the partition directly, we will recursively construct a partition $\{S_\alpha:\alpha<\kappa\}$ of $\kappa$ and then define $X_\beta=\bigcup\{K_\alpha:\alpha\in S_\beta\}$ for each $\beta<\kappa$. Our recursion will be done in steps of length $\omega$. So let $\gamma<\kappa$ be a limit and assume that we have constructed $\{S_\alpha:\alpha<\gamma\}$ in such a way that the following properties hold for all $\beta<\gamma$:

\begin{itemize}
\item[(i)] $0<|S_\beta|<\kappa$,
\item[(ii)] $\beta\in\bigcup\{S_\alpha:\alpha\leq\beta\}$, and
\item[(iii)] if $x\in D\cap X_\beta$, then $T(x)\subset\bigcup\{S_\alpha:\alpha\leq\beta+1\}$.
\end{itemize}

Let $s(\gamma)=\min\{\alpha<\kappa:\alpha\notin\bigcup\{X_\beta:\beta<\gamma\}\}$ and let $S_\gamma=\{s(\gamma)\}$, notice that conditions $(i)$ and $(ii)$ hold for $\beta=\gamma$. Now assume that $\{S_{\gamma+n}:n\leq k\}$ have been defined satisfying conditions (i), (ii) and (iii) (where meaningful). Since $\kappa$ is regular, $F=\bigcup\{T(x):x\in D\cap X_{\gamma+k}\}$ is a set of cardinality strictly less than $\kappa$. Let $s(\gamma+k+1)=\min\{\alpha<\kappa:\alpha\notin F\cup(\bigcup\{S_\beta:\beta\leq\gamma+k\})\}$ and define 
$$
S_{\gamma+k+1}=(F\setminus\bigcup\{S_\beta:\beta\leq\gamma+k\})\cup\{s(\gamma+k+1)\}.
$$
Then it is easy to see that conditions (i), (ii) and (iii) hold for this step of the construction. Thus, we can carry out our construction through all steps.

By the discussion at the begining of the proof and Lemma \ref{decreasing}, we only have to prove that if $\gamma<\kappa$ is a limit ordinal and $x\in D$, then $\{\rho(x,A_\alpha):\alpha<\gamma\}$ converges (as a net) to $\rho(x,A_\gamma)$.

If $x\in A_\gamma$, then $x\in A_\alpha$ for all $\alpha<\gamma$ so $\{\rho(x,A_\alpha):\alpha\leq\gamma\}$ is constant $0$ and the convergence is trivial.

Now assume that $x\notin A_\gamma$. This means that there is $\beta<\gamma$ such that $x\in X_\beta$. By our construction, $T(x)\subset\bigcup\{S_\alpha:\alpha\leq\beta+1\}$. Thus, if $\beta+2<\alpha\leq\gamma$, this implies that $\rho(x,A_\alpha)\geq r(x)$ and by $(\ast)$, we obtain that in fact $\rho(x,A_\alpha)= r(x)$. So $\{\rho(x,A_\alpha):\beta+2\leq\gamma\}$ is constant equal to $r(x)$. Thus in this case the convergence also holds.

This completes the proof that $\A$ is a closed subspace of $\pair{CL(X),\wijs{\rho}}$ that is homeomorphic to $\kappa$ with the order topology.
\end{proof}

\begin{coro}\label{embeddingcoro}
Let $\kappa$ be a regular uncountable cardinal. If $\pair{X,\rho}$ is a locally separable metric space of weight $\kappa$, the space $\pair{CL(X),\wijs{\rho}}$ is not paracompact. 
\end{coro}

Now consider a singular cardinal $\lambda$ with a discrete metric $\rho$. Since the weight of the Wijsman hyperspace is at most $\lambda$, the ordinals that can be embedded in the Wijsman hyperspace are less than or equal to $\lambda$. In order to obtain a non-paracompactness result as in Corollary \ref{embeddingcoro}, we are interested in what cardinals of uncountable cofinality can be embedded in the Wijsman hyperspace.

\begin{ex}
For every two infinite cardinals $\lambda<\kappa$ there is a discrete metric space $\pair{X,\rho}$ with $|X|=\kappa$ such that $\pair{CL(X),\wijs{\rho}}$ has a closed copy of the ordered space $\lambda$.
\end{ex}
\begin{ex2}
Let $X$ be of cardinality $\kappa$, $W\in[X]\sp{\lambda}$ and define a metric $\rho$ on $X$ in such a way that
$$
\rho(x,y)=\left\{
\begin{array}{cc}
0,& \textrm{ if }x=y,\\
1,& \textrm{ if }x\neq y\textrm{ and }\{x,y\}\subset W,\\
2,& \textrm{ if }x\neq y\textrm{ and }\{x,y\}\cap (X\setminus W)\neq\emptyset.
\end{array}
\right.
$$
for all $x,y\in X$. Give an enumeration $W=\{w_\alpha:\alpha<\lambda\}$ and define $\psi:\lambda\to CL(X)$ by $\psi(\beta)=\{w_\alpha:\beta\leq\alpha\}$. It is easy to see that $\psi$ is continuous. By Lemma \ref{decreasing} it easily follows that $\psi$ is closed.
\end{ex2}

\begin{ques}
Given a singular cardinal $\lambda$, does there exist a discrete metric $\rho$ on $\lambda$ such that $\pair{CL(\lambda),\wijs{\rho}}$ has no closed copies of $\kappa$ for all regular $\kappa<\lambda$?
\end{ques}

Finally, we know nothing about normality of the Wijsman hyperspace when the weight of the base space is a singular cardinal, even if the base space is discrete.

\begin{ques}
Does there exist a locally separable (or discrete, in particular) metric space $\pair{X,\rho}$ such that $\pair{CL(X),\wijs{\rho}}$ is normal?
\end{ques}

\section{Isolated points}

In section 2 of the paper \cite{cao-jun-moors}, the authors study when the finite subsets of a metric space can be isolated in the Wijsman hyperspace. In \cite[Example 2.3]{cao-jun-moors} the authors construct a countable discrete metric space $\pair{X,\rho}$ such that every non-empty finite subset is isolated in $\pair{CL(X),\wijs{\rho}}$ and they ask whether it is possible to do this with sets of arbitrary cardinality \cite[Question 2.4]{cao-jun-moors}. The following result answers this question in the affirmative. Recall that a discrete metric space $\pair{X,\rho}$ is \emph{uniformly discrete} if there is $\epsilon>0$ such that $\rho(x,y)>\epsilon$ for all $x,y\in X$ with $x\neq y$.

\begin{thm}\label{solutionisolated}
Let $\pair{X,\rho}$ be a uniformly discrete and bounded metric space. Then there exists a metric space $\pair{Y,\rho\sp\prime}$ such that $X\subset Y$, $|Y|=|X|$, $\rho\sp\prime\!\!\restriction_{X\times X}=\rho$ and each $F\in[Y]\sp{<\omega}\setminus\{\emptyset\}$ is isolated in $\pair{CL(Y),\wijs{\rho\sp\prime}}$.
\end{thm}
\begin{proof}
If $X$ is finite, let $X=Y$ and $\rho\sp\prime=\rho$. So assume that $X$ is infinite. Fix $0<\eta<1$ be such that for every $x,y\in X$ with $x\neq y$, $\eta\leq\rho(x,y)$. Let $M>0$ be any number such that $2M\geq\sup\{\rho(x,y):x,y\in X\}$.

Let $\{X_n:n<\omega\}$ be a pairwise disjoint collection of sets of cardinality $|X|$ such that $X_0=X$. Call $Y_n=\bigcup\{X_i:i\leq n\}$ for each $n<\omega$ and let $Y=\bigcup\{X_i:i<\omega\}$. For each $n<\omega$, give an injective enumeration $X_{n+1}=\{x(n,F):F\in[Y_n]\sp{<\omega}\}$. Let $\{\eta_n:n<\omega\}\subset(0,\infty)$ be a strictly decreasing sequence such that $\eta_0=\eta$.

We define $\rho\sp\prime:Y\times Y\to(0,\infty)$ as follows
$$
\rho\sp\prime(p_0,p_1)=\left\{
\begin{array}{ll}
\rho(p_0,p_1), & \textrm{ if }p_0,p_1\in X\\
0, & \textrm{ if } p_0=p_1,\\
M, & \textrm{ if }\exists\ \! i\in\{0,1\}, n<\omega,F\in[Y_n]\sp{<\omega}\textrm{ with }\\ & \ p_i=x(n,F), p_{1-i}\in Y_{n+1}\setminus F,\\
M+\eta_n, & \textrm{ if }\exists\ \!i\in\{0,1\}, n<\omega,F\in[Y_n]\sp{<\omega}\textrm{ with }\\ & \ p_i=x(n,F), p_{1-i}\in F.
\end{array}
\right.
$$
In order to prove that $\rho\sp\prime$ is a metric, it suffices to prove the triangle inequality (the rest of the proof is easy). So let $p,q,r\in Y$. 
 If $\{p,q,r\}\subset X$ then by our hypothesis, $\rho$ restricted to $X$ is a metric and we know that the triangle inequality holds for this space. So assume that $r\in X_{m+1}$ for some $m<\omega$, we also know that $r=x(m,H)$ for some $H\in[Y_m]\sp{<\omega}$.

Let $s=\rho\sp\prime(p,q)$. By the definitions of $\eta$ and $M$ and the construction of $\rho\sp\prime$, we know that 
$$
(\ast)\ \eta_m<s\leq 2M.
$$

Then we have three cases to consider.

\vskip6pt
\noindent{\it Case 1: $p,q\notin H$.}
\vskip6pt
In this case the triangle $pqr$ has sides $M$, $M$ and $s$. Then the inequalities to be checked are: $M\leq M+s$, which is clear; and $s\leq M+M$ which is true by property $(\ast)$.

\vskip6pt
\noindent{\it Case 2: $p\in H$ and $q\notin H$.}
\vskip6pt
In this case the triangle $pqr$ has sides $M+\eta_m$, $M$ and $s$. The inequalities to be checked are: $M<(M+\eta_m)+s$, which is clear; $(M+\eta_m)<M+s$, which follows from $(\ast)$; and $s\leq 2M+\eta_m$, which follows again from $(\ast)$.

\vskip6pt
\noindent{\it Case 3: $p,q\in H$.}
\vskip6pt
In this case the triangle $pqr$ has sides $M+\eta_m$, $M+\eta_m$ and $s$. The inequalities to be checked are: $(M+\eta_m)\leq (M+\eta_m)+s$, which is clear; and $s\leq (M+\eta_m)+(M+\eta_m)$, which follows from $(\ast)$.\vskip6pt

This proves that $\rho\sp\prime$ defines a metric. Now, finally, let $F\in[Y]\sp{<\omega}$. Let $k=\min\{n<\omega:F\subset Y_n\}$ and consider the point $x(k,F)\in X_{k+1}$.  We next prove that
$$
\{A\in CL(Y):F\subset A,\rho\sp\prime(x(k,F),A)>M+\eta_{k+1}\}=\{F\},
$$
which implies that $F$ is isolated in $\pair{CL(Y),\wijs{\rho}}$.

Let $z\in Y$, it is enough to prove that $\rho\sp\prime(x(k,F),z)>\eta_{k+1}$ if and only if $z\in F$. If $z\in F$ then $\rho\sp\prime(x(k,F),z)=M+\eta_k>M+\eta_{k+1}$. If $z\in Y_{k+1}\setminus F$, then $\rho\sp\prime(x(k,F),z)=M<M+\eta_{k+1}$. Otherwise, $z=x(l,G)$ for some $l\geq k+1$ and $G\subset Y_l$. If $x(k,F)\in G$, then $\rho\sp\prime(x(k,F),z)=M+\eta_l\leq M+\eta_{k+1}$. If $x(k,G)\notin G$, then $\rho\sp\prime(x(k,F),z)=M<M+\eta_{k+1}$. Thus, we have finished the proof. 
\end{proof}

Recall that if $X$ is a set with the metric $\rho$ such that $\rho(x,y)=1$ if $x\neq y$ (this is called the $0-1$ metric), then $\pair{CL(X),\wijs{\rho}}$ is homeomorphic to the space $2\sp X\setminus\{\mathbf{0}\}$ (see \cite[Example 2.1]{cao-jun-moors}). Thus we obtain the following corollaries.

\begin{coro}
For every cardinal $\kappa\neq 0$ there is a metric space $\pair{Y,\rho}$ such that $[Y]\sp{<\omega}\setminus\{\emptyset\}$ is an open and discrete subset of $\pair{CL(Y),\wijs{\rho}}$.
\end{coro}

\begin{coro}\label{changes}
For every infinite cardinal $\kappa$ there is a metric space $\pair{Y,\rho}$ and a subset $X\subset Y$ such that $|X|=|Y|=\kappa$ and $[X]\sp{<\omega}\setminus\{\emptyset\}$ is an open discrete subset of $\pair{CL(Y),\wijs{\rho}}$ but its closure in $\pair{CL(X),\wijs{\rho\restriction_{X\times X}}}$ is homeomorphic to the space ${}\sp{\kappa}{2}\setminus\{0\}$.
\end{coro}

If we start with the $0-1$ metric and apply Theorem \ref{solutionisolated}, we obtain a metric $\rho\sp\ast$. By the proof of Theorem \ref{solutionisolated}, we can construct $\rho\sp\ast$ with range exactly $\{1\}\cup\{1+\frac{1}{n+1}:n<\omega\}$. Notice further that $\rho\sp\ast$ is \emph{not} an ultrametric because there are triangles with all their sides of different length. We will next see that this metric $\rho\sp\ast$ is, in some sense, the most simple metric that can be obtained. First let's see that we cannot obtain all finite subsets isolated when working with finite metrics.

\begin{lemma}\label{isolatedfiniteslemma}
If $\pair{X,\rho}$ is a metric space and $F\in[X]\sp{<\omega}\setminus\{\emptyset,X\}$ is isolated in $\pair{CL(X),\wijs{\rho}}$, then there are $y_0,\ldots,y_m\in X$ and $\epsilon_0,\ldots,\epsilon_m\in(0,\infty)$ such that $X\setminus\bigcup\{B_\rho(y_i,\epsilon_i):i\leq m\}=F$.
\end{lemma}
\begin{proof}
Notice that $X$ embeds in $CL(X)$ with the function $x\mapsto\{x\}$ so every point of $F$ is isolated in $X$. Since $F$ is isolated, there is a basic open neighborhood of $F$ of the form $U=\bigcap\{\U(x_i,\alpha_i,\beta_i):i\leq k\}$, where $k<\omega$, $F\subseteq\{x_i:i<k\}$ and such that $U=\{F\}$. We may assume that there is $r\leq k$ such that $F=\{x_i:i<k\}$. 
Notice that $X$ embeds in $CL(X)$ with the function $x\mapsto\{x\}$ so every point of $F$ is isolated in $X$. So for each $i<r$, 
Let $0<\gamma_i<\beta_i$ be such that $B_\rho(x_i,\gamma_i)=\{x_i\}$. Define $V=\bigcap\{\U\sp-(x_i,\gamma_i):i<r\}$.

First consider the case that $r=k$.  Then $F\in V\subset U$ so in fact $V=\{F\}$. Notice that if $z\in X$, then by the definition of $V$, $F\cup\{z\}\in V$ so it follows that $z\in F$. This means that $X=F$ which is a contradiction to our hypothesis.

So $r<k$, let $m=k-r$ and $y_i=x_{r+i}$ for each $i< m$. For each $i< m$, let $\epsilon_i=\rho(F,y_i)>0$. So we just have to prove that
$$
X\setminus\bigcup\{B_\rho(y_i,\epsilon_i):i< m\}=F.
$$

By the choice of $\epsilon_i$ for $i< m$ we have that the right side of the equation is contained in the left side. Then let $z\in X\setminus\bigcup\{B_\rho(y_i,\epsilon_i):i\leq m\}$, we will prove that $F\cup\{z\}\in U$. 

If $i<r$, then clearly $\rho(F\cup\{z\},x_i)=0\in(\alpha_i,\beta_i)$ so $F\cup\{z\}\in\U(x_i,\alpha_i,\beta_i)$. If $r\leq i< k$, notice that $\rho(z,x_i)\geq\epsilon_i$. However since $\rho(F,x_i)=\epsilon_i$ and $F$ is finite, this distance is attained at some point of $F$. This means that $\rho(p,x_i)=\min\{\rho(y,x_i):y\in F\}=\epsilon_i$ for some $p\in F$. So then $\rho(F\cup\{z\})=\min\{\rho(y,x_i):y\in F\cup\{z\}\}=\epsilon_i$ because $F\cup\{z\}$ is also finite. So $F\cup\{z\}\in\U(x_i,\alpha_i,\beta_i)$ in this case as well.

Thus $F\cup\{z\}\in U$ and $U=\{F\}$ so $z\in F$. This proves the other inclusion and we have finished the proof.
\end{proof}

\begin{propo}\label{decreasingsequence}
If $\pair{X,\rho}$ is a discrete metric space such that $\{\rho(x,y):x,y\in X\}$ contains no infinite strictly decreasing sequences and every finite subset of $X$ is isolated in $\pair{CL(X),\wijs{\rho}}$, then $X$ is finite.
\end{propo}
\begin{proof}
Aiming towards a contradiction, let us assume that $X$ is infinite. We will recursively construct $\{F_i:i<\omega\}\subset[X]\sp{<\omega}$ as follows: Let $p\in X$ be chosen arbitarily and define $F_0=\{p\}$. By Lemma \ref{isolatedfiniteslemma} applied to $F_0$, it easily follows that $\{\rho(p,x):x\in X\}$ is bounded, let $\delta(p)>0$ be any bound.

Now assume that we have constructed $F_0\subset\ldots\subset F_k$. By Lemma \ref{isolatedfiniteslemma} there are $y_0,\ldots,y_m\in X$ and $\delta(y_i)$ for each $i\leq m\in(0,\infty)$ such that $X\setminus\bigcup\{B_\rho(y_i,\delta(y_i)):i\leq m\}=F_k$. Let $F_{k+1}=\{y_0,\ldots,y_m\}\cup F_k$.

This completes the construction of the $F_k$. Now define an ordering on $\bigcup\{F_k:k\in \omega\}$ as follows. For each $k<\omega$ and each $x\in F_{k+1}\setminus F_k$, choose $p(x)\in F_{k}\setminus F_{k-1}$ such that $x\in B_\rho(p(x),\delta(p(x)))$. Since the balls chosen around the points of $F_{k}\setminus F_{k-1}$ cover $X\setminus F_{k-1}$, there is such a $p(x)$. Define a tree ordering $\lhd$ on $\bigcup \{F_k:k\in \omega\}$ so that $F_k\setminus F_{k-1}$ is the $k^{\text th}$ level of the tree, and for each $y\in F_k\setminus F_{k-1}$ the set of successors of $y$ is $\{x\in F_{k+1}\setminus F_k: y=p(x)\}$. 

This is an infinite tree with finite levels so by K\"onig's Lemma (see \cite[Lemma III.5.6]{kunen-set-theory-2011}) it has an infinite branch $\{x_k:k<\omega\}$ with $p=x_0$ and $x_k\lhd x_{k+1}$ for $k<\omega$. Notice that for each $0<k<\omega$, we have that $x_{k-1}\not\in B(x_k,\delta(x_k))$ but, on the other hand $x_{k+1}\in B(x_{k},\delta(x_{k}))$. Therefore, if $0<k<\omega$, $\rho(x_{k+1},x_k)<\rho(x_k,x_{k-1})$, contradicting our assumptions on $\rho$. Thus, $X$ is finite.
\end{proof}

\begin{coro}
If $\pair{X,\rho}$ is a metric space with $\rho$ finite-valued and every finite subset of $X$ is isolated in $\pair{CL(X),\wijs{\rho}}$, then $X$ is finite.
\end{coro}

So in fact Proposition \ref{decreasingsequence} does not only rule out finite metrics but also metrics that do not contain decreasing sequences. So in this sense $\rho\sp\ast$ is the most simple such metric.

The next result rules out ultrametrics in uncountable sets from the possible spaces in which finite subsets are isolated in the Wijsman hyperspace. We are motivated to consider ultrametrics by the paper \cite{bertacchi-costantini} in which the authors study disconnectedness properties of ultrametric spaces.

\begin{propo}\label{ultraisctble}
Let $\pair{X,\rho}$ be an ultrametric space such that every non-empty finite subset of $X$ is isolated in $\pair{CL(X),\wijs{\rho}}$. Then $X$ is countable.
\end{propo}
\begin{proof}
By Lemmas \ref{ultralemma} and \ref{isolatedfiniteslemma}, the following property is easy to see.
\vskip6pt
\noindent$(\ast)$ If $x\in X$, there are pairwise disjoint $B_0,\ldots, B_m\in\B$ with $X\setminus\{x\}=\bigcup\{B_i:i\leq m\}$.
\vskip6pt
By recursion, we will construct a tree $\pair{T,\lhd}$ whose elements are pairs $\pair{x,r}$, where $x\in X$ and $r\in(0,\infty)$. Start choosing any $x_0\in X$, by property $(\ast)$ it easily follows that $\rho$ is bounded so there is $r_0\in(0,\infty)$ such that $X=B_\rho(x_0,r_0)$; then $\pair{x_0,r_0}$ is the smallest element of $T$. 

Assume that $\pair{x,r}\in T$, let us construct the immediate successors of this element. By property $(\ast)$, there are clopen balls $B_0,\ldots,B_m\in\B$ such that $X\setminus\{x\}=\bigcup\{B_i:i\leq m\}$. By Lemma \ref{ultralemma}, some of these clopen balls are contained in $B_\rho(x,r)$ and the rest miss it. If $B_\rho(x,r)\cap B_i=\emptyset$ for all $i\leq m$, then we stop the construction, so $\pair{x,r}$ has no successors. Otherwise, we may assume that $B_i\subset B_\rho(x,r)$ if and only if $i\leq n$ for some $n\leq m$. Let $B_i=B_\rho(x_i,r_i)$ for $i\leq n$. Then the immediate successors of $\pair{x,r}$ are exactly $\{\pair{x_i,r_i}:i\leq n\}$.

Our tree will only have levels at stages $<\omega$ so what we have said is enough to define the tree. By the construction, it is easy to see that the following properties hold.
\begin{itemize}
\item[(i)] Every node of $T$ has finitely many successors.
\item[(ii)] If $\pair{x,r},\pair{y,s}\in T$ and $\pair{x,r}\lhd\pair{y,s}$, then $s<r$ and $B_\rho(y,s)\subset B_\rho(x,r)$.
\end{itemize}

We would like to prove that $X=\{x:\exists r\in(0,\infty)\ (\pair{x,r}\in T)\}$, this would clearly prove that $X$ is countable. We can recursively define $T_0=\{\pair{x_0,r_0}\}$ and $T_{n+1}$ to be the set of all immediate successors of elements of $T_n$, for each $n<\omega$. Then by induction on $n<\omega$, it is easy to see that
$$
X=\{x:\exists m<n, r\in(0,\infty)\ (\pair{x,r}\in T_m)\}\cup(\bigcup\{B_\rho(x,r):\pair{x,r}\in T_n\})
$$
where $\{x:\exists m<n, r\ (\pair{x,r}\in T_m)\}$ is a finite (empty if $n=0$) set and  $\{B_\rho(x,r):\pair{x,r}\in T_n\}$ are pairwise disjoint.

Thus, if there is $p\in X\setminus\{x:\exists r\in(0,\infty)\ (\pair{x,r}\in T)\}$, then there is a branch $\{\pair{x_i,r_i}:i<\omega\}\subset T$ such that $\pair{x_i,r_i}\lhd\pair{x_{i+1},r_{i+1}}$ for all $i<\omega$. By property $(\ast)$, $X\setminus\{p\}=\bigcup\{C_0,\ldots,C_k\}$, where $k<\omega$ and $\{C_0,\ldots,C_k\}\subset \B$ are pairwise disjoint.

For each $i<\omega$ there is $s(i)\leq k$ such that $x_i\in C_{s(i)}$. If $i<j<\omega$, then by Lemma \ref{ultralemma} and the facts that $p\in B_\rho(x_j,r_j)\setminus C_{s(i)}$ and $x_i\in C_{s(i)}\setminus B_\rho(x_j,r_j)$, the set $C_{s(i)}$ must be disjoint from $B_\rho(x_j,r_j)$ so $s(i)\neq s(j)$. But then we are constructing an injective function $i\mapsto s(i)$ from $\omega$ to $k$, this is a contradiction.

Thus, such situation is impossible. This implies that $X=\{x:\exists r\in(0,\infty)\ (\pair{x,r}\in T)\}$ so we have finished the proof.
\end{proof}

However, the countable Example 2.3 in \cite{cao-jun-moors} can be essentially replaced by an ultrametric space.

\begin{ex}
There is a countable infinite ultrametric space $\pair{X,\rho}$ such that every finite subset of $X$ is isolated in $\pair{CL(X),\wijs{\rho}}$.
\end{ex}
\begin{ex2}
Let $X=\omega$ and
$$
\rho(m,n)=\left\{
\begin{array}{ll}
1+1/(\min\{m,n\}+1),&\textrm{ if }m\neq n\\
0,&\textrm{ if }m=n.
\end{array}\right.
$$
It easily follows that $\pair{X,\rho}$ is a $1$-discrete ultrametric space. Let $F\in[X]\sp{<\omega}\setminus\{\emptyset\}$, we next prove that $\{F\}$ is an open set. Let $z=(\max{F})+1$, $F=\{x_0,\ldots,x_m\}$ and $(z+1)\setminus F=\{y_0,\ldots,y_n\}$. Then
\begin{eqnarray}
\{F\}&=&\{A\in CL(X):\rho(x_i,A)<1 \textrm{ for }i\leq m, \rho(y_j,A)>1 \textrm{ for }j\leq n,\nonumber\\ & & \textrm{ and }\rho(z,A)>1+1/{(z+1)}\}\nonumber
\end{eqnarray}
is an open set.
\end{ex2}

We finally consider Question 3.3 of \cite{cao-jun-moors}. In that question, the authors ask whether discrete metric spaces have $0$-dimensional Wijsman hyperspaces. In \cite{bertacchi-costantini} there is an example of a discrete ultrametric space whose Wijsman hyperspace is not $0$-dimensional. In fact, we can say a little more.

\begin{ex}
For each cardinal $\omega\leq\kappa\leq\cont$, there is a uniformly discrete ultrametric space of cardinality $\kappa$ with its Wijsman hyperspace not $0$-dimensional.
\end{ex}
\begin{ex2}
Let $X\subset[1,2]$ be a dense subspace of cardinality $\cont$ and let $\rho$ be the metric defined as $\rho(x,y)=\max{\{x,y\}}$ whenever $x\neq y$. Then $\pair{X,\rho}$ is discrete and ultrametric. That $\pair{CL(X),\wijs{\rho}}$ is not $0$-dimensional follows from Theorem 19 of \cite{bertacchi-costantini}.
\end{ex2}

\end{document}